\documentclass[11pt]{amsart}
\usepackage{amsmath, amsthm, mathabx,amscd, amsfonts, amssymb, hyperref, mathrsfs, textcomp, epsfig, a4wide, graphicx, IEEEtrantools, tikz, verbatim, xypic}

\usetikzlibrary{matrix, decorations.pathreplacing,angles,quotes}

\newtheorem{thm}{Theorem}
\newtheorem{question}{Question}
\newtheorem{prop}{Proposition}
\newtheorem{lemma}{Lemma}
\newtheorem{cor}{Corollary}
\theoremstyle{definition}
\newtheorem{defn}{Definition}

\theoremstyle{remark}
\newtheorem{remark}{Remark}
\newtheorem{example}{Example}

     \RequirePackage{rotating}                   
    \def\HSt{%
       \setbox0=\hbox{$\widehat{\mathit{HS}}$}
       \setbox1=\hbox{$\mathit{HS}$}
       \dimen0=1.1\ht0
       \advance\dimen0 by 1.17\ht1
       \smash{\mskip2mu\raise\dimen0\rlap{%
          \begin{turn}{180}
              {$\widehat{\phantom{\mathit{HS}}}$}
           \end{turn}} \mskip-2mu    
                \mathit{HS}
    }{\vphantom{\widehat{\mathit{HS}}}}{}}

     \RequirePackage{rotating}                   
    \def\HMt{%
       \setbox0=\hbox{$\widehat{\mathit{HM}}$}
       \setbox1=\hbox{$\mathit{HM}$}
       \dimen0=1.1\ht0
       \advance\dimen0 by 1.17\ht1
       \smash{\mskip2mu\raise\dimen0\rlap{%
          \begin{turn}{180}
              {$\widehat{\phantom{\mathit{HM}}}$}
           \end{turn}} \mskip-2mu    
                \mathit{HM}
    }{\vphantom{\widehat{\mathit{HM}}}}{}}

\newcommand{\HMf}{\widehat{\mathit{HM}}}
    \newcommand{\HSb}{\overline{\mathit{HS}}}
\newcommand{\HSf}{\widehat{\mathit{HS}}}
\newcommand{\spin}{\mathfrak{s}}
\newcommand{\ztwo}{\mathbb{F}}
\newcommand{\Pin}{\mathrm{Pin}(2)}

\newcommand{\Rin}{\mathcal{R}}
\newcommand{\V}{\mathcal{V}}

\newcommand{\Arf}{\mathrm{Arf}}

\begin{document}

\title{Manolescu correction terms and knots in the three-sphere} 

\begin{abstract}
Manolescu correction terms are numerical invariants of homology three-spheres arising from $\Pin$-equivariant Seiberg-Witten theory that contain information about homology cobordism. We discuss several constraints on these invariants for homology spheres obtained by Dehn surgery on a knot in the three-sphere (and, more generally, in an integral homology $L$-space) in terms of the surgery coefficient, the concordance order, and the genus.
\end{abstract}

\author{Francesco Lin}
\address{Department of Mathematics, Princeton University and School of Mathematics, Institute for Advanced Study} 
\email{fl4@math.princeton.edu}
\maketitle

A basic question in low-dimensional topology is the following (see Problem $4.2$ in Kirby's list \cite{Kir}): which homology three-spheres $Y$ bound a homology ball? When are two homology spheres homology cobordant? $\Pin$-equivariant Seiberg-Witten theory has recently been shown to be very powerful when addressing questions of this kind. Manolescu (\cite{Man}) has employed it to define a package of homological invariants of rational homology spheres, which he used to disproof the long standing Triangulation Conjecture. Among the key features of the theory there are the \textit{Manolescu correction terms} of a homology sphere $Y$,
\begin{equation*}
\alpha(Y)\geq \beta(Y)\geq \gamma(Y).
\end{equation*}
These are integral lifts of the Rokhlin invariant $\mu(Y)$, and are invariant under homology cobordism. Indeed, using them Manolescu proved the following statement (which is equivalent to the Triangulation Conjecture being false in high dimensions by \cite{Gal} and \cite{Mat}): in the homology cobordism group $\Theta_3^H$ there are no $2$-torsion elements with Rokhlin invariant one.
\par
The author (\cite{Lin}) has introduced the analogue of these tools in the framework Kronheimer and Mrowka's of monopole Floer homology (\cite{KM}). This approach works for every three-manifold, and the two constructions are conjectured to provide the same invariants in the case of rational homology spheres; this said, in the present paper Manolescu correction terms will denote the numerical invariants arising from the latter approach.
\\
\par
In general not much is known about the structure of the homology cobordism group other that it is not finitely generated (\cite{Fur}, \cite{FS} and more recently \cite{Sto2}). While it is known that every homology sphere is homology cobordant to a hyperbolic one (\cite{Mye}), $\Pin$-techniques have been used in \cite{Sto} and \cite{Lin2} to show that there are homology spheres not homology cobordant to Seifert fibered spaces. As Floer theoretical tools have been shown to be effective in understanding problems in Dehn surgery (see among the others \cite{KMOS}, \cite{Gre} and \cite{HL}), the following is a natural question to ask. 
\begin{question}
Is every homology sphere homology cobordant to a homology sphere obtained by surgery on a knot in $S^3$?
\end{question}
This problem is also closely related to Problem $4.25$ in Kirby's list \cite{Kir}: if $W$ is a compact four-manifold homotopy equivalent to $S^2$, is the generator of $H_2(W)$ always representable by a \textsc{pl}-embedded sphere? A way to find a couterexample would be to find a homology sphere not homology cobordant to any surgery on a knot and which bounds a homotopy $S^2$.
\\
\par
Manolescu correction terms seem a natural approach to study Question $1$, especially because the invariants involved fit in various surgery exact triangles (see \cite{Lin2}). In the present paper, we discuss several constraints on them for homology spheres obtained by Dehn surgery on the three-sphere, and in general any integral homology $L$-space, in terms of other well-studied quantities in knot theory. Recall that an $L$-space is a rational homology sphere whose reduced monopole Floer homology vanishes; any connected sum of of Poincar\'e homology spheres (with either orientation) is an integral $L$-space, and it is conjectured that these are indeed all the examples which are integral homology spheres. We will also exhibit some specific examples of spaces so that the upcoming results can be rephrased in the form \textit{``there exists a $Y$ which is not homology cobordant to any homology sphere such that etc.''}. We say that a surgery on a knot is \textit{even}, \textit{odd}, \textit{positive} or \textit{negative} if the surgery coefficient is the reciprocal of a integer with that property. Our first result involves even surgeries.
\begin{thm}\label{even}
Suppose $Y'$ is obtained by even surgery on a knot in an integral $L$-space $Y$. If the surgery is positive, we have $\alpha(Y')=\beta(Y')=\delta(Y)$ and
\begin{equation*}
\gamma(Y')=
\begin{cases}
\delta(Y')\text{ if }\delta(Y')\text{ is even}\\
\delta(Y')-1\text{ otherwise.}
\end{cases}
\end{equation*}
If the surgery is negative we have $\beta(Y')=\gamma(Y')=\delta(Y)$ and
\begin{equation*}
\alpha(Y')=
\begin{cases}
\delta(Y')\text{ if }\delta(Y')\text{ is even}\\
\delta(Y')+1\text{ otherwise.}
\end{cases}
\end{equation*}
In particular, if $Y'$ is obtained by even surgery on an integral $L$-space two of its Manolescu correction terms coincide.
\end{thm}
Here $\delta(Y)$ is the correction term arising in monopole Floer homology as one half of the bottom grading in the $U$-tower of $\HMt_{\bullet}(Y)$ (hence it is $-h(Y)$ where $h(Y)$ is the Fr\o yshov invariant defined in \cite{KM}, and corresponds to $d(Y)/2$ in Heegaard Floer theory \cite{OSd}). Its value on $S^3$ is zero. As suggested by Manolescu \cite{Man}, these correction terms induce concordance invariants of a knot $K$ by taking the associated branched double cover (with respect to the unique spin structure) in the same spirit of \cite{MO}. We call these invariants $\alpha(K),\beta(K)$, $\gamma(K)$ and $\delta(K)$ respectively (in our normalization $\delta(K)$ is $1/4$ of the same named invariant in \cite{MO}). A particularly interesting class of knots to study is that of Whitehead doubles (cf. Problem $1.38$ in Kirby's list \cite{Kir}). Unfortunately, the previous result implies that this new correction terms do not carry new information in this sense.  

\begin{cor}\label{whitehead}
For a knot $K\subset S^3$, let $Wh(K)$ be the Whitehead double with a positive clasp. Then $\alpha(Wh(K))=\beta(Wh(K))=0$, and
\begin{equation*}
\gamma(Wh(K))=\begin{cases}
\delta(K)\text{ if }\delta(K)\text{ is even}\\
\delta(K)-1\text{ otherwise.}
\end{cases}
\end{equation*}
\end{cor}
Recall that for a homology sphere $Y'$ of the form $Y_{1/m}(K)$ for $Y$ an integral homology $L$-space the invariant $\delta(Y')$ is determined by $K$ and the sign of $m$, see for example \cite{NW}. In the case of Manolescu correction terms, also the parity of $m$ has to be taken into account. When dealing with odd surgeries, the restrictions are not as strong as in Theorem \ref{even}, but we can nevertheless prove the following.

\begin{prop}\label{odd}
Consider a homology sphere $Y'$ of the form $Y_{1/m}(K)$ with $m$ odd, for $Y$ an integral homology $L$-space. If $K$ has Arf invariant zero, the correction terms depend only on the sign of $m$. If $K$ has Arf invariant $1$, then for $m> 0$, $\alpha$ and $\beta$ of $Y_{1/m}(K)$ are independent of $m$, while for $m< 0$, $\beta$ and $\gamma$ are independent of $m$.
\end{prop}
We expect all the correction terms for an odd surgery on an Arf invariant $1$ knot to depend only on the sign of the surgery. Unfortunately the techniques of this paper are not enough to show that, the difference between in the Arf invariant zero and one cases being in the shape of the relevant surgery exact triangle (\cite{Lin2}).
\\
\par
A class of knots which has received considerable interest is that of knots in $S^3$ with concordance order two, which forms a quite large collection (see for example \cite{HKL}). In this case, exploiting the symmetries of the invariants one can prove the following.
\begin{thm}\label{order2}
Suppose $Y$ is obtained by surgery on a knot in $S^3$ of concordance order two. Then $\delta(Y)$ and $\beta(Y)$ have the same sign, where we consider zero to have the same sign as every other integer.
\end{thm}

In general, one can exploit the wide literature of computations available in Heegaard Floer homology (via the isomorphism with usual monopole Floer homology, see \cite{HFHM1}, \cite{CGH}, and subsequent papers). This was used for example in \cite{Lin2} to discuss $(\pm1)$-surgeries on alternating knots with Arf invariant one, following the computations of \cite{OSalt}. In general, the same approach can be used to compute the correction terms of surgeries on alternating and $L$-space knots. Among the many applications one could point out, we would like to address the following consequence of the results in \cite{Gai}.
\begin{thm}\label{genus}
There exists a constant $C>0$ such that the following holds. Suppose $Y$ is a \textrm{connected sum} of homology spheres $Y_1\hash\dots\hash Y_n$, each $Y_i$ obtained from an integral $L$-space by surgery on a knot with genus at most $G$. Then $\alpha(Y)-\gamma(Y)\leq C\cdot G$.\end{thm}

\vspace{0.3cm}
\textit{Organization of the paper. }In Section \ref{review} we briefly recall the basic definitions and formal properties of $\Pin$-monopole Floer homology that we will need. Section \ref{tors} is dedicated to some simple consequences of the Gysin exact sequence and the role of $U$-torsion. Section \ref{exact} studies in detail the properties of the surgery exact triangle for a knot in an integral $L$-space. Finally, in Section \ref{examples} we discuss some explicit examples.
\\
\par
\textit{Acknowledgements. }The present paper was inspired by some nice conversations with Matthew Hedden and Mark Powell regarding Corollary \ref{whitehead}, whom I thank together with the organizers of the workshop \textit{Synchronizing Smooth and Topological 4-Manifolds} at BIRS, which made those conversations possible. The author is grateful to Paolo Aceto for suggesting Figure \ref{composition}, and to Jen Hom for the nice conversations regarding Question $1$. This work was partially supported by the NSF grant DMS-0805841, the Shing-Shen Chern Membership Fund and the IAS Fund for Math.

\vspace{0.5cm}
\section{A quick review of $\Pin$-monopole Floer homology}\label{review}
Most of the results of the present paper follow from some formal properties of $\Pin$-monopole Floer homology, which we now review. We refer the reader to \cite{Lin} for a more thorough introduction and \cite{Lin3} for the actual construction. In what follows, $\ztwo$ denotes the field with two elements.
\par
Given a closed connected oriented three-manifold $Y$, there is a natural involution by conjugation $\jmath$ on the set of spin$^c$ structures $\mathrm{Spin}^c(Y)$. The fixed points, which we call \textit{self-conjugate}, corresponds to spin$^c$ structures induced by a spin structure. To each self-conjugate spin$^c$ structure $\spin$ we associate the $\Pin$-\textit{monopole Floer homology groups}, which in the long exact sequence
\begin{equation}\label{pin2}
\cdots\stackrel{i_*}{\longrightarrow}\HSt_{\bullet}(Y,\spin)\stackrel{j_*}{\longrightarrow} \HSf_{\bullet}(Y,\spin)\stackrel{p_*}{\longrightarrow} \HSb_{\bullet}(Y,\spin)\stackrel{i_*}{\longrightarrow}\cdots
\end{equation}
and are called respectively \textit{HS-to}, \textit{HS-from} and \textit{HS-bar}. These groups carry a relative $\mathbb{Z}$-grading which can be lifted to an absolute $\mathbb{Q}$-grading. They also carry a structure of graded module over the ring
\begin{equation*}
\Rin= \ztwo[[V]][Q]/(Q^3)
\end{equation*}
where the actions of $V$ and $Q$ have degree respectively $-4$ and $-1$. Furthermore, there are analogous cohomological version which satisfy a version of Poincar\'e Lefschetz duality: for example, up to grading shift,
\begin{equation*}
\HSt_{\bullet}(Y,\spin)\equiv\HSf_{\bullet}(\bar{Y},\spin),
\end{equation*}
where $\bar{Y}$ denotes the manifold with orientation reversed.
\\
\par
For any rational number $d$ let $\V_d$ and  $\V^+_d$ be the graded $\ztwo [[V]]$-modules $\ztwo[V^{-1},V]]$ (the ring of Laurent power series) and $\ztwo[V^{-1},V]]/V\ztwo [[V]]$ where the grading is shifted so that the element $1$ has degree $d$. We have the identifications as absolutely graded $\Rin$-modules:
\begin{align*}
\HSt_{\bullet}(S^3)&\cong \V^+_2\oplus \V^+_1\oplus \V^+_0\\
\HSf_{\bullet}(S^3)&\cong \Rin\langle-1\rangle\\\
\HSb_{\bullet}(S^3)&\cong \V_2\oplus \V_1\oplus \V_0.
\end{align*}
The action of $Q$ on the \textit{to} and \textit{bar} groups is an isomorphism from the first tower onto the second tower and from the second tower onto the third (and zero otherwise). More generally, given a rational homology sphere $Y$ and a self-conjugate spin$^c$ structure $\spin$ we have an isomorphism of $\Rin$-modules
\begin{equation*}
\HSb_{\bullet}(Y,\spin)\cong \HSb_{\bullet}(S^3)
\end{equation*}
up to grading shift. The group $\HSt_{\bullet}(Y,\spin)$ vanishes in degrees low enough, and the map $i_*$ is an isomorphism in degrees high enough. Hence $i_*\left(\HSb_{\bullet}(Y,\spin)\right)$, considered as an $\ztwo[[V]]$-module, decomposes as the direct sum
\begin{equation*}
\V^+_c\oplus \V^+_b\oplus \V^+_a.
\end{equation*}
We call these three summands respectively the $\gamma$, $\beta$ and $\alpha$-towers. The action of $Q$ sends the $\gamma$-tower onto the $\beta$-tower and the $\beta$-tower onto the $\alpha$-tower. Manolescu's correction terms are then defined to be numbers
\begin{equation*}
\alpha(Y)\geq\beta(Y)\geq\gamma(Y)
\end{equation*}
such that
\begin{equation*}
a=2\alpha(Y),\quad b= 2\beta(Y)+1, \quad c=2\gamma(Y)+2.
\end{equation*}
The inequalities between these quantities follow from the module structure. The invariants satisfy the following properties:
\begin{enumerate}
\item they reduce to $-\mu(Y,\spin)$ modulo two, where the latter is the Rokhlin invariant;
\item when $Y$ is an integral homology sphere, they are integers;
\item they are invariant under spin rational homology cobordism;
\item $\alpha(\bar{Y})=-\gamma(Y)$, $\beta(\bar{Y})=-\beta(Y)$ and $\gamma(\bar{Y})=-\alpha(Y)$, where $\bar{Y}$ is $Y$ with the orientation reversed.
\end{enumerate}
These are the quantities we are interested in studying in this paper.
\\
\par
If $[\spin]\in \mathrm{Spin}^c(Y)/\jmath$ is the orbit of a pair of non isomorphic spin$^c$ structures $\spin\neq\bar{\spin}$, we can define
\begin{equation*}
\HSt_{\bullet}(Y,[\spin]):=\HMt_{\bullet}(Y,\spin)\equiv\HMt_{\bullet}(Y,\bar{\spin}).
\end{equation*}
This module is only relatively graded, and can be given the structure of $\Rin$-module where $Q$ acts by zero and $V$ acts by $U^2$. We define the total Floer group
\begin{equation*}
\HSt_{\bullet}(Y)=\bigoplus_{[\spin]\in\mathrm{Spin}^c(Y)/\jmath} \HSt_{\bullet}(Y,[\spin])
\end{equation*}
A cobordism $W$ from $Y_0$ to $Y_1$ induces a map of $\Rin$-modules
\begin{equation*}
\HSt_{\bullet}(W): \HSt_{\bullet}(Y_0)\rightarrow \HSt_{\bullet}(Y_1).
\end{equation*}
This maps decomposes along according to pairs of conjugate spin$^c$ structures on the cobordism, and furthermore this assignment is functorial under compositions.

\vspace{0.5cm}
\section{The Gysin exact sequence and $U$-torsion}\label{tors}
A very important feature  of the theory is the Gysin exact sequence
\begin{equation*}
\cdots \stackrel{\pi_*}{\longrightarrow} \HSt_{\bullet}(Y)\stackrel{\cdot Q}{\longrightarrow}\HSt_{\bullet}(Y)\stackrel{\iota_*}{\longrightarrow} \HMt_{\bullet}(Y)\stackrel{\pi_*}{\longrightarrow} \HSt_{\bullet}(Y)\stackrel{\cdot Q}{\longrightarrow} \cdots,
\end{equation*}
see Chapter $4$ of \cite{Lin}. This is an exact sequence of $\Rin$-modules where on $\HMt_{\bullet}$ we have that $V$ acts as $U^2$ and $Q$ acts as zero. Furthermore this sequence is compatible with the maps induced by cobordisms. This can be used to infer information about $\HSt_{\bullet}$ from the knowledge of $\HMt_{\bullet}$. In the case $Y$ is a homology sphere, in degrees high enough it is four periodic, and the basic block it looks like
\begin{center}
\begin{tikzpicture}
\matrix (m) [matrix of math nodes,row sep=1.5em,column sep=1em,minimum width=2em]
  {\cdot &\cdot &\cdot\\
  \ztwo &\ztwo& \ztwo\\
  \ztwo&0&\ztwo\\
  \ztwo &\ztwo&\ztwo\\};
  \path[-stealth]
  (m-2-1) edge node [above]{} (m-2-2)
      (m-4-2) edge node [above]{} (m-4-3)
       (m-2-3) edge node [above]{} (m-3-1)
        (m-3-3) edge node [above]{} (m-4-1)
  ;
\end{tikzpicture}
\end{center}
where the central column represents $\HMt_{\bullet}$ while the other two represent $\HSt_{\bullet}$. Among the other things, in \cite{Lin2} we proved the following.
\begin{lemma}
An integral homology sphere $Y$ is an $L$-space, i.e. $\HMt_{\bullet}(Y)\equiv \HMt_{\bullet}(S^3)$ as $\ztwo[[U]]$-graded modules, up to grading shift, if and only if $\HSt_{\bullet}(Y)\equiv \HSt_{\bullet}(S^3)$ as $\Rin$-graded modules, up to grading shift. In particular for an integral $L$-space we have $\alpha(Y)=\beta(Y)=\gamma(Y)=\delta(Y)$.
\end{lemma}

The following result appears in \cite{Sto3} and as the proof only appeals to formal properties of the Gysin exact sequence, it readily adapts to our setting. We quickly review it for completeness.
\begin{lemma}\label{dp1}
For a homology sphere $Y$, consider the quantity
\begin{equation*}
\Delta=\frac{1}{2}\left(\mathrm{min}\{\mathrm{gr}(x)\lvert x\in \gamma\textrm{-tower}, x\not\in\mathrm{Im}Q\}-2\right).
\end{equation*}
Then we have
\begin{equation*}
\Delta(Y)=\begin{cases}
\delta(Y)\text{ if }\delta(Y)\equiv\mu(Y)(2)\\
\delta(Y)-1\text{ otherwise.}
\end{cases}
\end{equation*}
\end{lemma}
\begin{proof}
Suppose $x$ is in the $\gamma$-tower and not in the image of $Q$. The Gysin exact sequence implies that $\iota_*x\neq 0$, and indeed the compatibility with the sequence for the \textit{bar}-versions of the invariants implies that $\iota_*x$ is in the $U$-tower of $\HMt_{\bullet}(Y)$. Hence $\Delta(Y)\leq \delta(Y)$. Analogously, if $y$ in in the $U$-tower in degree congruent to $2+2\mu(Y)$ modulo $4$, then the fact that in degrees high enough the sequence is standard implies that $y=\iota_*x$ for $x$ in the $\gamma$-tower. The exactness implies that $x$ is not in the image of $Q$, so that $\Delta(Y)\geq \delta(Y)-1$. The result then follows because $\Delta$ has the same parity as $\mu(Y)$.\end{proof}

We now discuss some simple bounds on the invariants in terms of the $U$-torsion of $\HMt_{\bullet}(Y)$. Recall that for a homology sphere $Y$ we have the isomorphism of $\ztwo[U]$-modules
\begin{equation*}
\HMt_{\bullet}(Y)=\left(\ztwo[U^{-1},U]]/U\ztwo[U]\right)\bigoplus\left(\oplus\ztwo[U]/U^{n_i}\right)
\end{equation*}
for some collection of $n_i\geq 1$. Define $t(Y)$ to be the maximum of the $n_i$. This is readily seen to be an invariant of the homology sphere. We have the following properties:
\begin{enumerate}
\item $t(\bar{Y})=t(Y)$;
\item $t(Y\hash Y')=\mathrm{max}\{t(Y),t(Y')\}$.
\end{enumerate}
The first property follows from Poincar\'e duality, while the second property is a consequence of the isomorphism\begin{equation*}
\HMf_{\bullet}(Y\hash Y')=\mathrm{Tor}^{\ztwo[[U]]}_{*,*}(\HMf_{\bullet}(Y),\HMf_{\bullet}(Y'))[1],
\end{equation*} 
(see \cite{Lin4} for discussion of the unpublished work of Bloom, Mrowka and Ozsv\'ath, and \cite{OS2} for the analogous result in Heegaard Floer homology) and the fact that, forgetting about the gradings,
\begin{align*}
\mathrm{Tor}^{\ztwo[[U]]}_{*,*}(\ztwo[U]/U^{n},\ztwo[U]/U^{m})&=\ztwo[U]/U^{n}\oplus \ztwo[U]/U^{n}\text{ for }n\leq m\\
\mathrm{Tor}^{\ztwo[[U]]}_{*,*}(\ztwo[[U]],\ztwo[U]/U^{n})&=\ztwo[U]/U^{n}.
\end{align*}

\begin{lemma}
There exists a constant $C'>0$ such that for every homology sphere $\alpha(Y)-\gamma(Y)\leq C't(Y)$.
\end{lemma}
\begin{proof}
The key point is to show that for every $Y$ there is an inequality $\delta(Y)-\gamma(Y)\leq C'' t(Y)$. Indeed, if this holds, we have
\begin{equation*}
\alpha(Y)-\delta(Y)=\delta(\bar{Y})-\gamma(\bar{Y})\leq C''t(\bar{Y})=C''t(Y),
\end{equation*}
where the first equality follows from the properties of the correction terms under orientation reversal and the second follows from property $(1)$ above. The Gysin exact sequence implies that the elements in the $\gamma$-tower which have degree less than $2\delta(Y)$ are in the image of $Q$. As $\alpha\geq \delta$, the portion of the $\beta$-tower in degree less than $2\delta(Y)$ is in the image of $\pi_*$, and as the $V$-action is mapped to a $U^2$-action it comes from some torsion subgroup of $\HMt_{\bullet}$. This implies that for some constant $C'''>0$ we have
\begin{equation*}
\max\{0,\delta(Y)-\beta(Y)\}\leq C''' t(Y).
\end{equation*} 
Similarly, for the portion of the $\gamma$-tower in degrees less than $\mathrm{min}(2\delta(Y),2\beta(Y))$ is also in the image of $\pi_*$, hence it also provides a linear lower bound on $t(Y)$. The claim follows by summing these two lower bounds.
\end{proof}

\begin{proof}[Proof of Theorem \ref{genus}]
This follows from the previous proposition, the observations about $t(Y)$ and Theorem $3$ in \cite{Gai}, which for our purposes implies the following: if $Y$ is obtained by surgery on a knot of genus $g$, then $t(Y)\leq 2g$ (while the author only explicitly talks about $S^3$, his results hold for any integral homology $L$-space). Of course, here we use the isomorphism between monopole and Heegaard Floer homology (see \cite{HFHM1}, \cite{CGH} and subsequent papers).
\end{proof}

\vspace{0.5cm}

\section{The surgery exact triangle in an $L$-space}\label{exact}
In this section we study the properties of the surgery exact triangle in $\Pin$-monopole Floer homology in the case of a knot in an integral $L$-space. We first recall the main result of \cite{Lin2}. Suppose $Z$ is a manifold with torus boundary, and consider three oriented curves $\gamma_i$, $i=1,2,3$ so that
\begin{equation*}
\gamma_i\cdot \gamma_{i+1}=-1.
\end{equation*}
Denote by $Y_i$ the three manifold obtained by Dehn filling $Z$ along $\gamma_i$. There is a natural cobordism $W_i$ from $Y_i$ to $Y_{i+1}$ obtained by a single $2$-handle attachment. Of these three cobordisms, exactly two are spin, and we can assume after relabeling that these are $W_1$ and $W_2$. Then there is a well-defined map of $\Rin$-modules
\begin{equation*}
\check{F}_3: \HSt_{\bullet}(Y_3)\rightarrow \HSt_{\bullet}(Y_1)
\end{equation*}
so that the triangle
\begin{center}
\begin{tikzpicture}
\matrix (m) [matrix of math nodes,row sep=2em,column sep=1.5em,minimum width=2em]
  {
  \HSt_{\bullet}(Y_1)&& \HSt_{\bullet}(Y_2)\\
  &\HSt_{\bullet}(Y_3)&\\};
  \path[-stealth]
  (m-1-1) edge node [above]{$\HSt_{\bullet}(W_1)$} (m-1-3)
   (m-2-2) edge node [left]{$\check{F}_3$}(m-1-1)
  (m-1-3) edge node [right]{$\HSt_{\bullet}(W_2)$} (m-2-2)  
  ;
\end{tikzpicture}
\end{center}
is exact. The analogous statement is true for the \textit{from} and \textit{bar} versions.
\par
We are particularly interested in the case in which $Z$ is the complement of a knot $K$ in a homology sphere $Y$, and the surgery triple is of the form
\begin{equation*}
1/p,\quad 0,\quad 1/(p+1)
\end{equation*}
(with respect to the Seifert framing). We denote by $\check{S}_{1/q}$ the group $\HSt_{\bullet}(Y_{1/q}(K))$. The surgery exact triangle specializes to the triangle
\begin{center}
\begin{tikzpicture}
\matrix (m) [matrix of math nodes,row sep=2em,column sep=1.5em,minimum width=2em]
  {
  \check{S}_{1/q} && \check{S}_{0}\\
  &\check{S}_{1/(q+1)} &\\};
  \path[-stealth]
  (m-1-1) edge node [above]{$\check{A}^s_q$} (m-1-3)
   (m-2-2) edge node {}(m-1-1)
  (m-1-3) edge node [right]{$\check{B}^s_{q+1}$} (m-2-2)  
  ;
\end{tikzpicture}
\end{center}
where the maps $\check{A}^s_q$ and $\check{B}^s_q$ are those induced by the spin cobordisms. We use an analogous notation for the \textit{from} and \textit{bar} versions. Also, we will consider the usual monopole Floer counterpart: we will denote the groups $\HMt_{\bullet}(Y_{1/q}(K))$ by $\check{M}_{1/q}$ and the maps by $\check{A}^m_q$ and $\check{B}^m_q$ respectively. We can now state our first observation regarding the surgery exact triangle. In \cite{KMOS} it is shown that the composition 
\begin{equation}\label{composite}
\check{B}^m_{q}\circ \check{A}^m_q:\check{M}_{1/q} \rightarrow \check{M}_{1/q} 
\end{equation}
is zero. This follows from exhibiting an embedded sphere of self-intersection zero and applying a border case of the adjunction inequality. This is the key observation when making inductive arguments involving the surgery exact triangle. In the $\Pin$ case, we have the following.
\begin{lemma}\label{key}
The composite (\ref{composite}) is given by multiplication by $Q$. 
\end{lemma}
\begin{proof}
The composition of the two cobordisms is described by the Kirby diagram in Figure \ref{composition2}. The cobordism from $Y_{1/q}(K)$ to $Y_0(K)$ is given by a two handle attachment along the knot $K'$, while the following one from $Y_0(K)$ to $Y_{1/q}(K)$ is given by attaching another two handle along a zero framed meridian of $K'$. If we trade this second handle for a $1$-handle (i.e. adding a dot in the notation of \cite{GS}), we obtain a pair of canceling $1$ and $2$-handles. In particular, the composite cobordism is obtained from the product one $[0,1]\times Y_{1/q}(K)$ by removing a neighborhood $S^1\times D^3$ of a loop and replacing it by $S^2\times D^2$. The result then readily follows from the fact that the map induced by $S^2\times S^2$ with two ball removed induces multiplication by $Q$ (see the proof of Theorem $5$ in \cite{Lin2}).
\end{proof}

\begin{figure}
  \centering
\def\svgwidth{0.9\textwidth}
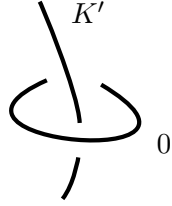
\caption{A handlebody description of the composite of the cobordisms defining the map $\check{B}_{q-1}\circ \check{A}_q$. This link is inside $Y_{1/q}$.}
    \label{composition2}
\end{figure}

Recall that as $Y_0(K)$ is a homology $S^1\times S^2$ we have that
\begin{equation*}
\bar{M}_0\equiv \ztwo[U^{-1},U]]\langle-1\rangle\oplus \ztwo[U^{-1},U]].
\end{equation*}
We call the image in $\check{M}_0$ of the two summands respectively the bottom and top $U$-towers. We record the following simple observation.
\begin{lemma}\label{usualkey}
For $q\geq 0$, the image of $\check{A}^m_q$ coincides with the bottom tower of $\check{M}_0$.
\end{lemma}
\begin{proof}
The result is obvious for $\check{A}^m_0$. This implies in particular that elements not in the bottom tower have non trivial image under $\check{B}^m_1$. Now the result follows by induction, using the fact that $\check{B}^m_{q}\circ \check{A}^m_q$ is zero.
\end{proof}

\vspace{0.3cm}
In the $\Pin$-setting, as discussed in \cite{Lin2} there is a big qualitative difference in the shape of the surgery exact triangle in the cases where the knot has Arf invariant zero or one. This follows from the relation between the $\Pin$ invariants and the Rokhlin invariant. If $\Arf(K)=0$, we have the identification
\begin{equation}\label{arf0}
\bar{S}_0= \ztwo[V^{-1},V]][Q]/(Q^3)\langle-1\rangle \oplus \ztwo[V^{-1},V]][Q]/(Q^3).
\end{equation}
This decomposes as a $\ztwo[[V]]$-module as a direct sum of six modules $\V$, and we call the image of the ones in the first summand the bottom $\alpha$, $\beta$ and $\gamma$ towers, and the ones in the second summand the top $\alpha$, $\beta$ and $\gamma$ towers. If $\Arf(K)=1$ we have the identification as $\ztwo[[V]]$-modules
\begin{equation}\label{arf1}
\bar{S}_0= (\V\langle1\rangle\oplus \V)\oplus (\V\langle1\rangle\oplus \V)\langle 2\rangle
\end{equation}
and the action of $Q$ is an isomorphism from the summand to the second and from the third summand to the fourth. Their image in $\check{S}_0$ will be called respectively the bottom $\gamma$ and $\beta$ towers and the top $\beta$ and $\alpha$ towers. Intuitively, the main difference between the $\Arf(K)=1$ case is that the top $\gamma$-tower cancels with the bottom $\alpha$-tower (cf. \cite{Lin2}).
\\
\par
In the case of a knot with Arf invariant $1$, the exact triangle in the \textit{bar} version looks like
\begin{center}
\begin{tikzpicture}
\matrix (m) [matrix of math nodes,row sep=0.5em,column sep=2em,minimum width=2em]
  { & \vdots &\\
  \ztwo & \ztwo& \ztwo\\
  \ztwo & \ztwo&\cdot \\
  \ztwo & \ztwo& \ztwo\\
  \cdot & \ztwo&\ztwo \\
  \ztwo & \ztwo& \ztwo\\
  & \vdots & \\
  };
  \path[-stealth]
  (m-2-2) edge node {}(m-2-3)  
  (m-2-1) edge node {}(m-3-2)
  (m-3-1) edge node {}(m-4-2)
  (m-5-2) edge node {}(m-5-3)
  (m-6-2) edge node {}(m-6-3)
  ;
\end{tikzpicture}
\end{center}
extended in both directions in a four-periodic fashion. The groups are from left to right $\bar{S}_{1/p}$, $\bar{S}_0$ and $\bar{S}_{1/{p+1}}$, and the depicted maps are $\bar{A}^s_q$ and $\bar{B}^s_q$ (which have degree respectively $-1$ and $0$). The third map is an isomorphism from the $\gamma$-tower of $\bar{S}_{1/{p+1}}$ to the $\alpha$- tower of $\bar{S}_{1/p}$, and is given by the multiplication by a power series in $V$ with leading term $1$. Notice that the maps $\bar{A}^s_q$ depend on the parity of $q$ once one takes into account the grading modulo four. This is related to the fact that the Rokhlin invariant changes by doing odd surgery ona knot with Arf invariant $1$. The bottom $\beta$ and $\gamma$-towers correspond to the image of $\bar{A}^s_q$ for $q$ even, while the top $\alpha$ and $\beta$-towers correspond to the image of $\bar{A}^s_q$ for $q$ odd.
\par
The $\Arf(K)=0$ is a little trickier. This is because, unlike the $\Arf(K)=1$ case, the identification (\ref{arf0}) above is not canonical, and for a fixed identification the exact triangle in the \textit{bar} version could look either as
\begin{center}
\begin{tikzpicture}
\matrix (m) [matrix of math nodes,row sep=0.5em,column sep=1.5em,minimum width=2em]
  {

  \ztwo & \cdot& \ztwo&\ztwo\\
  \ztwo & \ztwo&\ztwo&\ztwo \\
  \ztwo & \ztwo&\ztwo& \ztwo\\
  \cdot & \ztwo&\cdot \\
  };
  \path[-stealth]
	  (m-1-1) edge node {}(m-2-2)  
	  (m-2-1) edge node {}(m-3-2)  
	  (m-3-1) edge node {}(m-4-2)  
	  
	  (m-1-3) edge node {}(m-1-4)  
	  (m-2-3) edge node {}(m-2-4)  
	  (m-3-3) edge node {}(m-3-4)

  ;
\end{tikzpicture}
\end{center}
or
\begin{center}
\begin{tikzpicture}
\matrix (m) [matrix of math nodes,row sep=0.5em,column sep=1.5em,minimum width=2em]
  {

  \ztwo & \cdot& \ztwo&\ztwo\\
  \ztwo & \ztwo&\ztwo&\ztwo \\
  \ztwo & \ztwo&\ztwo& \ztwo\\
  \cdot & \ztwo&\cdot \\
  };
  \path[-stealth]
	  (m-1-1) edge node {}(m-2-2)  
	  (m-2-1) edge node {}(m-3-2)  
	  (m-3-1) edge node {}(m-4-2)  
	 	
	 (m-1-1) edge node {}(m-2-3)  
	  (m-2-1) edge node {}(m-3-3)

	  (m-1-3) edge node {}(m-1-4)  
	  (m-2-3) edge node {}(m-2-4)  
	  (m-3-3) edge node {}(m-3-4)  
	  
	  (m-2-2) edge[bend right] node {} (m-2-4)
	  (m-3-2) edge[bend right] node {} (m-3-4)
	  
  ;
\end{tikzpicture}
\end{center}
In both cases the two central columns correspond to $\bar{S}_0$ and the picture extend in both directions four-periodically. The second observation, which is a direct consequence of Lemma \ref{key}, is the following.
\begin{lemma}\label{maps}
Suppose we have chosen the identification (\ref{arf0}) so that $\bar{A}^s_0$ and $\bar{B}^s_1$ are as in the first case above. Then $\bar{A}^s_q$ and $\bar{B}^s_{q+1}$ are as is the first case if $q$ is even, and as in the second if $q$ is odd.
\end{lemma}
We will always that the identification (\ref{arf0}) is made according to this lemma.
\\
\par
In the case of usual Floer homology (see \cite{OSd}), we can define the correction terms for manifolds with $b_1=1$ by looking at the minimal gradings of the two $U$-towers. Focusing on the zero surgery on a knot $K$, this allows us to define the invariants
\begin{align*}
\delta_+(K)&=\frac{1}{2}\mathrm{min}\{\mathrm{gr}(x)\lvert x\in \text{top $U$-tower of }\check{M}_0\}\}\\
\delta_-(K)&=\frac{1}{2}\left(\mathrm{min}\{\mathrm{gr}(x)\lvert x\in \text{bottom $U$-tower of }\check{M}_0\}\}+1\right),
\end{align*}
which are normalized to be zero in $S^2\times S^1$. They correspond respectively to $(d_t-1/2)/2$ and $(d_b+1/2)/2$ in the Heegaard Floer context (recall that the grading conventions between the two theories differs by $b_1(Y)/2$).
\par
In the $\Pin$-setting, for the zero surgery on a knot with $\Arf(K)=0$ we can define the six correction terms $\alpha_{\pm},\beta_{\pm},\gamma_{\pm}$ as in the case of a homology sphere (see Section \ref{review}), where the \textit{bottom} correction terms (indicated with the minus) correspond to the left summand in (\ref{arf0}), while the top correspond to the quotient of $i_*(\bar{S}_0)$ by $\mathrm{Im}\check{A}^s_0\cap i_*(\bar{S}_0)$ (for which the notion of $\alpha$, $\beta$ and $\gamma$-towers is well defined). We also take account of the shift on the left summand and normalize bottom invariants by shifting their degrees up by $1$ as in the case of $\delta_-(K)$ above. In formulas:
\begin{align*}
\alpha_-(K)&=\frac{1}{2}\left(\mathrm{min}\{\mathrm{gr}(x)\lvert x\in \text{bottom $\alpha$-tower of }\check{S}_0\}\}+1\right)\\
\beta_-(K)&=\frac{1}{2}\left(\mathrm{min}\{\mathrm{gr}(x)\lvert x\in \text{bottom $\beta$-tower of }\check{S}_0\}\}\right)\\
\gamma_-(K)&=\frac{1}{2}\left(\mathrm{min}\{\mathrm{gr}(x)\lvert x\in \text{bottom $\gamma$-tower of }\check{S}_0\}\}-1\right)\\
\alpha_+(K)&=\frac{1}{2}\left(\mathrm{min}\{\mathrm{gr}(x)\lvert x\in \text{$\alpha$-tower of } i_*(\bar{S}_0)/\left(\mathrm{Im}\check{A}^s_0\cap i_*(\bar{S}_0)\right)\}\right)\\
\beta_+(K)&=\frac{1}{2}\left(\mathrm{min}\{\mathrm{gr}(x)\lvert x\in \text{$\beta$-tower of } i_*(\bar{S}_0)/\left(\mathrm{Im}\check{A}^s_0\cap i_*(\bar{S}_0)\right)\}-1\right)\\
\gamma_+(K)&=\frac{1}{2}\left(\mathrm{min}\{\mathrm{gr}(x)\lvert x\in \text{$\gamma$-tower of } i_*(\bar{S}_0)/\left(\mathrm{Im}\check{A}^s_0\cap i_*(\bar{S}_0)\right)\}-2\right)
\end{align*}
For example, for $S^2\times S^1$ all the six correction terms are zero.
\begin{remark}
The ambiguity choice of the identification (\ref{arf0}) could be a little misleading sometimes, and could lead to very different looking $\check{S}_0$. In particular the correction terms really depend on the choice of the identification. For example consider the module (written horizontally)
\begin{center}
\begin{tikzpicture}
\matrix (m) [matrix of math nodes,row sep=0.5em,column sep=1em,minimum width=2em]
  { 
 \cdot & \cdot &\cdot &\ztwo_{-1} &\ztwo &\ztwo&\cdot&\cdots\\
 \ztwo_{-4} &\ztwo &\ztwo &\cdot &\ztwo&\ztwo &\ztwo&\cdots\\
  };
  \path[-stealth]
 (m-1-5) edge[bend right] node {}(m-1-4)
 (m-1-6) edge[bend right] node {}(m-1-5) 
 (m-2-2) edge[bend left] node {}(m-2-1)
 (m-2-3) edge[bend left] node {}(m-2-2)
 (m-2-7) edge[bend left] node {}(m-2-6)
 (m-2-6) edge[bend left] node {}(m-2-5)
 ;
\end{tikzpicture}
\end{center}
given as a direct sum of towers, where the top row consists of the bottom towers and the bottom row consists of the top towers (for simplicity we have only indicated the $Q$ actions). For example, this is $i_*\bar{S}_0$ for the zero surgery on the torus knot $T(2,7)$. In this case, the bottom correction terms are all zero while the top correction terms are all $-2$. On the other hand we can do a change of basis, so that the module structure will be
\begin{center}
\begin{tikzpicture}
\matrix (m) [matrix of math nodes,row sep=0.5em,column sep=1em,minimum width=2em]
  { 
  \ztwo_{-4} &\ztwo&\cdot & \ztwo &\ztwo &\ztwo&\cdot&\cdots\\
  & & \ztwo & \cdot &\ztwo&\ztwo&\ztwo&\cdots\\
  };
\path[-stealth]
 (m-1-2) edge[bend right] node {}(m-1-1)
 (m-1-5) edge[bend right] node {}(m-1-4)
 (m-1-6) edge[bend right] node {}(m-1-5)
 (m-2-3) edge[bend left] node {}(m-1-2)
 (m-2-6) edge[bend left] node {}(m-2-5)
 (m-2-7) edge[bend left] node {}(m-2-6)
;
  
 \end{tikzpicture}
\end{center}
In this case we have $\alpha_-=0$, $\beta_-=\gamma_-=-2$ and $\alpha_+=\beta_+=0$ and $\gamma_-=-2$.
\end{remark}
\vspace{0.3cm}
In the case of surgery on a knot with $\Arf=1$ we can define the four correction terms: $\beta_-$ and $\gamma_-$ corresponding to the bottom towers, and $\alpha_+$ and $\beta_+$ corresponding to the top towers, by means of the same formulas above. Here we use the same conventions as in the previous case, so that for the zero surgery on the trefoil, for which we computed in \cite{Lin2} that
\begin{equation*}
\check{S}_0= (\V^+_1\oplus \V^+_0)\oplus (\V^+_{-1}\oplus \V^+_{-2}),
\end{equation*}
where $Q$ is an isomorphism from the first tower onto the second and from the third tower onto the forth, the bottom invariants are $0$ while the top invariants are $-1$. In both cases, we will denote the various correction terms as $\alpha_{\pm}(K), \beta_{\pm}(K)$ and $\gamma_{\pm}(K)$, and we can similarly define the correction terms $\delta_{\pm}(K)$ coming from $\check{M}_0$ (these are normalized so that they are both zero for the unknot).
\\
\par
The main observation is the following.
\begin{lemma}\label{zeros}
Suppose $K$ is a knot in an integral $L$-space $Y$. Then $\beta_-(K)=\gamma_-(K)=\delta(Y)$.
\end{lemma}
\begin{proof}
By hypothesis, $\HSt_{\bullet}(Y)\equiv \HSt_{\bullet}(S^3)\langle2\delta(Y)\rangle$. The existence of the $\Rin$-module homomorphism $\check{A}_0: \check{S}_{\infty}\rightarrow \check{S}_0$ implies that  $\beta_-(K)\geq\gamma_-(K)\geq\delta(Y)$. The existence of the $\Rin$-module homomorphism $\check{B}_{0}: \check{S}_0\rightarrow \check{S}_{\infty}$ (together with the description of $\bar{B}_0$ in the $\Arf=0$ case) implies that  $\delta(Y)\geq\beta_-(K)\geq\gamma_-(K)$, hence the result follows.
\end{proof}

With this in hand, we can proceed in the proof of Theorem \ref{even} and Proposition \ref{odd}. As the statements are invariant under orientation reversal, we can assume that the surgery coefficient is positive. We focus first on the case of a knot with Arf invariant zero, where the following proposition (which implies both Theorem \ref{even} and Proposition \ref{odd} in the Arf invariant zero case) holds.
\begin{prop}
Consider for $m>0$ the three-manifold $Y'= Y_{1/m}(K)$, where $\Arf(K)=0$ and $Y$ is an integral $L$-space. Then if $m$ is odd we have
\begin{equation*}
\alpha(Y')=\alpha_+(K),\quad\beta(Y')=\beta_+(K),\quad\gamma(Y')=\gamma_+(K)
\end{equation*}
while if $m$ is even we have $\alpha(Y')=\beta(Y')=\delta(Y)$ and
\begin{equation*}
\delta(Y)=
\begin{cases}
\delta_+(K)\text{ if }\delta_+(K)\equiv \delta(Y)\mathrm{ mod}(2)\\
\delta_+(K)-1\text{ otherwise.}
\end{cases}
\end{equation*}
\end{prop}
\begin{proof}
The statement is invariant under an overall grading shift, so that we can assume without loss of generatily that $\delta(Y)$ is zero. The proof proceeds by induction on the surgery coefficient. As by hypothesis $\check{S}_{\infty}$ is isomorphic (up to grading shift) to $\HSt_{\bullet}(S^3)$, the map $\check{A}^s_0$ is a surjection onto the bottom tower hence the result for $S^3_1(K)$ readily follows by the exact triangle and the definition of the correction terms of the knot $K$.
\par
We want now to compute the correction terms for $Y_{1/2}(K)$. The exact triangle and the module structure implies that $i_*(\bar{S}_{1/2})$ is a quotient $\Rin$-module of
\begin{equation}\label{quoT}
T=i_*(\bar{S}_0)/ \check{A}_1(i_*\bar{S}_1).
\end{equation}
The maps $\bar{A}^s_1$ and $\bar{B}^s_2$ are determined by Lemma \ref{maps}. Notice that because of Lemma \ref{zeros}, the part of $i_*(\bar{S}_0)$ in degrees between $4k-1$ and $4k+2$ for $k\geq 0$ looks like one of the following
\begin{center}
(a)
\begin{tikzpicture}
\matrix (m) [matrix of math nodes,row sep=1em,column sep=0.5em,minimum width=2em]
  { \cdot& \ztwo&&\\
  \ztwo&\ztwo\\
  \ztwo&\ztwo\\
 \ztwo& \\
};

\path[-stealth]
 (m-2-1) edge[bend right] node {}(m-3-1)
 (m-3-1) edge[bend right] node {}(m-4-1)
 (m-1-2) edge[bend left] node {}(m-2-2)
(m-2-2) edge[bend left] node {}(m-3-2)

 ;

\end{tikzpicture}
\qquad
(b)
\begin{tikzpicture}
\matrix (m) [matrix of math nodes,row sep=1em,column sep=0.5em,minimum width=2em]
  { \cdot& \ztwo&&\\
  \ztwo&\ztwo\\
  \ztwo&\ztwo\\
 \cdot& \\
};

\path[-stealth]
 (m-2-1) edge[bend right] node {}(m-3-1)
 (m-1-2) edge[bend left] node {}(m-2-2)
(m-2-2) edge[bend left] node {}(m-3-2)

 ;

\end{tikzpicture}
\qquad
(c)
\begin{tikzpicture}
\matrix (m) [matrix of math nodes,row sep=1em,column sep=0.5em,minimum width=2em]
  { \cdot& \ztwo&&\\
  \ztwo&\ztwo\\
  \ztwo&\\
 \cdot& \\
};
\path[-stealth]
 (m-2-1) edge[bend right] node {}(m-3-1)
  (m-1-2) edge[bend left] node {}(m-2-2)
 (m-2-2) edge[bend left] node {}(m-3-1)
 
 ;

\end{tikzpicture}
\qquad
(d)
\begin{tikzpicture}
\matrix (m) [matrix of math nodes,row sep=1em,column sep=1em,minimum width=2em]
  { \cdot& \ztwo\\
  \ztwo&\\
  \ztwo&\\
 \cdot& \\
};
\path[-stealth]
 (m-1-2) edge[bend left] node {}(m-2-1)
  (m-2-1) edge[bend right] node {}(m-3-1)
 ;
\end{tikzpicture}
\end{center}
In these cases, the corresponding part of $i_*(\bar{S}_1)$ lying in degrees $4k-1$ and $4k+2$ is
\begin{center}
(a)
\begin{tikzpicture}
\matrix (m) [matrix of math nodes,row sep=1em,column sep=0.5em,minimum width=2em]
  { \ztwo&&\\
 \ztwo\\
  \ztwo\\
 \cdot \\
};
\path[-stealth]
 (m-1-1) edge[bend right] node {}(m-2-1)
 (m-2-1) edge[bend right] node {}(m-3-1)
 ;
\end{tikzpicture}
\qquad
(b)
\begin{tikzpicture}
\matrix (m) [matrix of math nodes,row sep=1em,column sep=0.5em,minimum width=2em]
  { \ztwo&&\\
 \ztwo\\
  \ztwo\\
 \cdot \\
};
\path[-stealth]
 (m-1-1) edge[bend right] node {}(m-2-1)
 (m-2-1) edge[bend right] node {}(m-3-1)
 ;
\end{tikzpicture}
\qquad
(c)
\begin{tikzpicture}
\matrix (m) [matrix of math nodes,row sep=1em,column sep=0.5em,minimum width=2em]
  { \ztwo&&\\
 \ztwo\\
  \cdot\\
 \cdot \\
};
\path[-stealth]
 (m-1-1) edge[bend right] node {}(m-2-1)
;
\end{tikzpicture}
\qquad
(d)
\begin{tikzpicture}
\matrix (m) [matrix of math nodes,row sep=1em,column sep=0.5em,minimum width=2em]
  { \ztwo&&\\
 \cdot\\
  \cdot\\
 \cdot \\
};
\end{tikzpicture}

\end{center}
In light of the description of the map $\bar{A}^s_1$, it is then straightforward to check that the module $T$ has in all four cases the shape
\begin{center}
\begin{tikzpicture}
\matrix (m) [matrix of math nodes,row sep=1em,column sep=0.5em,minimum width=2em]
  { \ztwo&&\\
 \ztwo\\
  \ztwo\\
 \cdot \\
};
\path[-stealth]
 (m-1-1) edge[bend right] node {}(m-2-1)
 (m-2-1) edge[bend right] node {}(m-3-1)
 ;

\end{tikzpicture}
\end{center}
in degrees between $4k-1$ and $4k+2$, for $k\geq 0$. In degree less than $1$, the fact that $\check{B}^s_1\circ\check{A}^s_1$ is multiplication by $Q$ implies that the only elements in $i_*\bar{S}_0$ which survive in the quotient are those in the $\gamma$-tower which are not in the image of $Q$. Putting these observations together, we have that if we can identify $i_*\bar{S}_2$ with $T$ then the result holds, where the identification of $\gamma$ with either $\delta$ or $\delta-1$ follows from Lemma \ref{dp1}. Hence, we are left to prove that no element of $T$ is killed by some element of $\check{S}_1$. For the elements in the $\gamma$-tower in degrees below zero, this follows from our observation above: if $\check{A}^s_1(x)$ is in the top $\gamma$-tower of $\check{S}_0$, then Lemma \ref{usualkey} implies that $Qx$ is in the $\gamma$-tower of $\check{S}_1$. In non-negative degrees the argument is slightly different between the cases $(a)$ and $(b)$ and the cases $(c)$ and $(d)$, and we first consider the latter. Suppose $y\in \check{S}_1$ kills the the corresponding element in the bottom $\beta$-tower (which coincides in these cases with the element in the top $\alpha$-tower). Because $\check{B}^s_1$ is zero on the latter, we have by Lemma \ref{key} that $Q\cdot y$ is zero. By the Gysin exact sequence for $Y_1(K)$, $y=\pi_*z$ for some $z\in \check{M}_1$, and furthermore $\check{A}^m_1(z)\neq 0$ because it has to map under $\pi_*$ to the bottom $\beta$-tower. Furthermore, $\check{A}^m_1(z)$ cannot belong to the bottom tower of $\check{M}_0$, so we get a contradiction to Lemma \ref{usualkey}. In the first two cases, suppose that $y\in\check{S}_1$ is an element such that $\check{A}^s_1$ is non-zero in the top $\alpha$-tower. Call the elements in the corresponding section of $i_*\bar{S}_1$ respectively $w$, $Q\cdot w$ and $Q^2w$. Then Lemma \ref{key} implies that $Q\cdot y= Q^2w$. In particular, $Q\cdot(Qw+y)$ is zero, hence $Qw+y$ is in the image of $\pi_*$. Furthermore, as $Qw$ is mapped to the sum of the element in the bottom $\beta$-tower and the element in the top $\alpha$-tower, we have that $\check{A}^s_1(Qw+y)$ is the element in the bottom $\beta$-tower. In case $(a)$, this is a contradiction because the latter element is not in the image of $\pi_*$, while in the case $(b)$ it is a contradiction again because of Lemma \ref{usualkey}.
\par
The proof now continues by induction, after one notices that because of Lemma \ref{maps} the map $\check{A}^s_2$, when restricted to $T$, is a map onto the bottom towers of $\check{S}_0$.
\end{proof}

The following proposition, together with the previous one, implies Theorem \ref{even} and Proposition \ref{odd}.

\begin{prop}
For a knot $K$ with Arf invariant $1$ in an integral $L$-space $Y$, consider for $m>0$ the three-manifold $Y'= Y_{1/m}(K)$. Then if $m$ is odd we have
\begin{equation*}
\alpha(Y)=\alpha_+(K),\quad\beta(Y)=\beta_+(K)
\end{equation*}
while if $m$ is even we have $\alpha(Y)=\beta(Y)=\delta(Y)$ and
\begin{equation*}
\delta(Y)=
\begin{cases}
\delta_+(K)\text{ if }\delta_+(K)\equiv \delta(Y)\mathrm{mod}(2)\\
\delta_+(K)-1\text{ otherwise.}
\end{cases}
\end{equation*}
\end{prop}
\begin{proof}
The proof of this result is the same as the one above, the main difference being that in the case of a knot with Arf invariant one, because of the shape of the triangle, the $\gamma$-tower is not determined by the zero surgery. The main observation is that even though in the odd case we cannot determine exactly the $\gamma$-tower, in the even case the elements of the $\gamma$-tower of $Y_{1/m}(K)$ in degree less than zero are (as in the Arf zero case) not in the image of $Q$ by Lemma \ref{key}, so that $\gamma$ is again determined by $\delta$ by Lemma \ref{dp1}.
\end{proof}
\vspace{0.3cm}
\begin{proof}[Proof of Corollary \ref{whitehead}]
It is shown in \cite{MO} that the branched double cover of $Wh(K)$ is the homology sphere obtained by $(1/2)$-surgery on $K\hash K^r$, $K^r$ being the knot with string orientation reversed. The result then follows from Theorem \ref{even}.
\end{proof}
\vspace{0.3cm}
The following weaker version of Theorem \ref{even} also holds more in general.
\begin{prop}
Suppose $Y$ is a homology sphere with $\alpha(Y)=\beta(Y)=\gamma(Y)=\delta(Y)$. If $Y'$ is obtained from $Y$ by even surgery, then $\beta(Y')=\delta(Y)$.
\end{prop}
\begin{proof}
As shown in Figure \ref{composition} there is a spin cobordism $W$ from $Y$ to $Y'$ with $b_2^+(W)=b_2^-(W)=1$. By Theorem $5$ in \cite{Lin2} we have the inequality
\begin{equation*}
\beta(Y')\geq \gamma(Y)=\delta(Y).
\end{equation*}
On the other hand, the reversed cobordism with the reversed orientation is still spin with $b_2^+(W)=b_2^-(W)=1$, so that again by Theorem $5$ in \cite{Lin2}
\begin{equation*}
\delta(Y)=\alpha(Y)\geq \beta(Y').
\end{equation*}
and the result follows.
\end{proof}

\begin{figure}
  \centering
\def\svgwidth{0.7\textwidth}
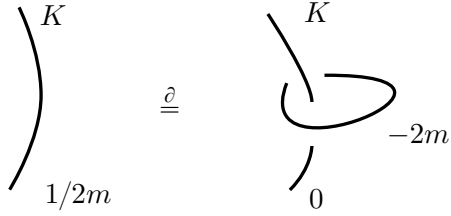
\caption{The boundary of these surgery diagrams are equivalent under a slam dunk move (see \cite{GS}). The diagram on the right describes a spin cobordism from $Y$ to $Y'$ with $b_2^+(W)=b_2^-(W)=1$.}
    \label{composition}
\end{figure} 
\vspace{0.3cm}

\begin{proof}[Proof of Theorem \ref{order2}]
The statement is clear if the surgery coefficient is even in light of Theorem \ref{even}. For the odd case, the key observation is that
\begin{equation}\label{simm}
\beta(S^3_{-1/n}(K))=\beta(\overline{S^3_{1/n}(\bar{K})})=-\beta(S^3_{1/n}(\bar{K}))=-\beta(S^3_{1/n}({K})).
\end{equation}
Here $\bar{K}$ denotes the mirror of $K$ and the first equality follows from the fact that the inputs are naturally identified, the second from the properties of $\beta$ and the third from the fact that, as $K$ has concordance order two, $\bar{K}$ is concordant to $K$, so that the corresponding surgeries are homology cobordant. Because the statement is invariant under orientation reversal, and in light of Proposition \ref{odd}, we can suppose that $Y$ is obtained by $+1$ surgery on a knot $K$. Suppose now that $\delta(Y)$ and $\beta(Y)$ has opposite signs, so that $\delta(Y)<0$ (because the surgery is positive) and $\beta(Y)>0$. Then $\beta_+(K)=\beta(Y)>0$. The module structure then implies that $\alpha_+(K)>0$. Now the symmetry (\ref{simm}) implies that $\beta(S^3_{-1}(K))<0$. On the other hand the bottom element of the $\beta$-tower of $S^3_{-1}(K)$ maps via $\check{A}^s_{-1}$ injectively into the top $\alpha$-tower of $S^3_0(K)$ (see again Lemma \ref{maps}), implying $\alpha_+(K)<0$, contradiction.
\end{proof}

\vspace{0.5cm}
\section{Some examples}\label{examples}
In this section we discuss some examples of homology spheres which turn to be handy when making statements of the form ''\textit{there exist homology spheres which are not homology cobordant to any homology sphere of the form etc.}''. Our main result, which is essentially a computation with the Eilenberg-Moore-spectral sequence \cite{Lin4}, is the following (cf. \cite{Sto2} for analogous results in the setting of $\Pin$-equivariant Seiberg-Witten Floer homology).
\begin{prop}\label{all}
For any $a\geq b\geq c$ integers with the same parity with $a-b\geq b- c$, there exists a homology sphere $Y$such that
\begin{equation*}
\alpha(Y)=\delta(Y)=a,\quad \beta(Y)=b,\quad \gamma(Y)=c.
\end{equation*}
\end{prop}
The more general geography problem seems to be much harder. For example, the author is not aware of any homology sphere for which $\delta\not\in\{\alpha,\alpha-1,\gamma+1,\gamma\}$, or a homology sphere for which $\alpha=\delta$ and $\alpha-\beta<<\beta-\gamma$. While one can easily construct candidate chain complexes with any given numerical invariants (satisfying the obvious inequalities $\alpha\geq \beta,\delta\geq \gamma$), it is not clear which of these can arise as the Floer chain complex of a homology sphere.
\\
\par
Consider the $\ztwo[[V]]$-module $\V_d(k)$ given by $\ztwo[[V]]/V^{k}\langle2k+d\rangle$. The degrees are shifted so that the element with minimum degree has degree $d$. We define the $\Rin$-module $M_k$ as
\begin{equation}\label{mk}
\V^+_0\oplus \V^+_{-4k+1}\oplus  \V^+_{-4k+2}\oplus \V_{-4k+3}(k)
\end{equation}
where the action of $Q$ is surjective from the second summand to the first, is an isomorphism from the third summand to the second and is injective from the fourth summand to the third. More graphically, this module is
\begin{center}
\begin{tikzpicture}
\matrix (m) [matrix of math nodes,row sep=0.1em,column sep=0.7em,minimum width=0.1em]
  {   \ztwo & \ztwo &\cdot&& \cdots &\cdots& \ztwo & \ztwo & \cdot & \ztwo_0 & \ztwo &\ztwo&\cdots \\
        &  & \ztwo & &  &&&  & \ztwo &  &  & &\\
};
\path[-stealth]
(m-2-3) edge[bend left] node {}(m-1-2)
(m-1-5) edge[bend right] node {}(m-1-1)
(m-1-6) edge[bend right] node {}(m-1-2) 
(m-1-2) edge[bend left] node {}(m-1-1)

(m-1-8) edge[bend left] node {}(m-1-7)
(m-1-11) edge[bend left] node {}(m-1-10)
(m-1-12) edge[bend left] node {}(m-1-11)
(m-2-9) edge[bend left] node {}(m-1-8)
(m-1-12) edge[bend right] node {}(m-1-8)
(m-1-11) edge[bend right] node {}(m-1-7)
;
\draw[decoration={brace,mirror,raise=-4pt,amplitude=15pt},decorate]
  (-5,-0.5) -- node[below=6pt] {$k$ copies} (2,-0.5);
\end{tikzpicture}
\end{center}
Here the bottom line represents the $\V_{-4k+3}(k)$ summand, hence has dimension $k$.
\begin{defn}
We say that a homology sphere $Y$ has \textit{simple type} $M_{k}$ if there is a decomposition as a direct sum of $\Rin$-modules
\begin{equation*}
\HSt_{\bullet}(Y)=M_k \oplus J
\end{equation*}
and there are not non-trivial Massey products between the two summands.
\end{defn}
The submodule $i_*(\HSb_{\bullet}(Y))$ consists of the three $\V^+$ summands in $M_k$. In particular we have that if $Y$ has simple type $M_{k}$ then
\begin{equation*}
\alpha(Y)=\delta(Y)=0\qquad\text{and}\qquad \beta(Y)=\gamma(Y)=-2k.
\end{equation*}
\begin{example}
Let $K_k$ be the torus knot $T(2, 8k-1)$. Then $Y_k=S^3_{-1}(K_k)$ has simple type $M_k$. To see this, recall the computations for alternating knots provided by \cite{Lin2} tell us that in the torsion spin$^c$ structure
\begin{equation*}
\HSt_{\bullet}(S^3_0(K_k),\spin_0)=\left(\V_{-1}\oplus \V_0\oplus \V_1\right)\oplus \left(\V_{-4n}\oplus \V_{-4n+1}\oplus \V_{-4n+2}\right)\oplus\ztwo^{m_k}\langle -4n\rangle
\end{equation*}
for some $m_k$ depending on $k$. Given the description of $\bar{A}^s_{-1}$ from Lemma \ref{maps}, the exact triangle implies a decomposition of the form (\ref{mk}). The fact that there are not non-trivial Massey products follows from the fact that in this case the extra summand comes for the other spin$^c$ structures on $S^3_0(K_k)$ and functoriality. This example readily generalizes for $(1/m)$-surgeries on $L$-space knots of Arf invariant $0$, with $m$ odd and negative.
\end{example}

The following is the key computation.
\begin{prop}\label{sum}
Suppose $Y$ and $Y'$ have simple type $M_k$ and $M_{k'}$  for $k'\leq k$ respectively. Then we have
\begin{equation*}
\alpha(Y\hash Y')=\delta(Y\hash Y')=0,\quad \beta(Y\hash Y')=-2k,\quad \gamma(Y\hash Y')=-2(k+k').
\end{equation*}
\end{prop}
\begin{proof}
In order to study the Eilenberg-Moore spectral sequence we need to deal with the \textit{from} groups. Poincar\'e duality tells us that up to grading shift and reversal
\begin{equation*}
\HSt^{\bullet}(Y)\cong\HSf_{\bullet}(\bar{Y}),
\end{equation*}
hence we can apply the result on connected sum for $\bar{Y}$ and $\bar{Y}'$ to the dual $\Rin$-modules $M^*_k\oplus J^*$ and $M^*_{k'}\oplus (J')^*$. For simplicity, we can suppose that the gradings in the latter are shifted so that the top degree elements of $M^*_k$ and $M^*_{k'}$ are zero: the final gradings will be straightforward to infer as by additivity we have $\delta(Y\hash Y')=0$, so we only need to focus on the differences. The key point is that while in general the $E^2$-page of the Eilenberg-Moore spectral sequence is rather complicated, the module $M_k^*$ has a two-step projective resolution
\begin{equation*}
M_k^*\leftarrow \Rin[1-4k]\oplus \Rin[0]\leftarrow \Rin[4k]
\end{equation*}
where the map on the right sends $1$ to $(Q,V^k)$. The hypothesis that the summands $J$ and $J'$ do not have Massey products into the towers imply that in order to study the correction terms we just need to understand $\mathrm{Tor}^{\Rin}_{*,*}(M^*_k,M^*_{k'})$. As the projective resolution has length two, there are no higher differentials in the spectral sequence, so that it collapses at the $E^2$-page. Recall that we have the identification
\begin{equation*}
\mathrm{Tor}^{\Rin}_{*,1}(M^*_k,M^*_{k'})=M^*_k\otimes_{\Rin}M^*_{k'}.
\end{equation*}
The latter can be written as a direct sum of two $\Rin$-modules, one of which is
\begin{equation}\label{towersum}
\ztwo[[V]]\langle0\rangle\oplus\ztwo[[V]]\langle-4k'+1\rangle\oplus\ztwo[[V]]\langle-4(k+k')+2\rangle\oplus\V_{3-4(k+k')}(k+k')\oplus\V_{2-4k'}(k')
\end{equation}
where the action of $Q$ is not trivial from one summand to the one next to it on the right when there are two $\ztwo$ summands that differ in degree by one, and the other is $\V_{5-4(k+k')}(k')$. The group $\mathrm{Tor}^{\Rin}_{*,2}$ is isomorphic to $\V_{2-4(k+k')}(k')$. In the simplest case where $k=k'=1$, $\mathrm{Tor}^{\Rin}_{*,*}$ can be graphically described as follows:
\begin{center}
\begin{tikzpicture}
\matrix (m) [matrix of math nodes,row sep=0.1em,column sep=1em,minimum width=0.1em]
  {    & \ztwo &\cdot&\cdot&  & \ztwo & \ztwo & \cdot & \ztwo & \ztwo &\ztwo&\cdots \\
        &  & \ztwo &\ztwo &  &\oplus&  & \ztwo &  &  & &\\
       {} & & & & & \ztwo&&&&{}&&{}\\
        &&&&&&&&\ztwo\\
};
\path[-stealth]
(m-1-6) edge[bend right] node {}(m-1-7)
(m-1-9) edge[bend right] node {}(m-1-10)
(m-1-10) edge[bend right] node {}(m-1-11)
(m-1-2) edge[bend right] node {}(m-2-3)
(m-1-7) edge[bend right] node {}(m-2-8)
(m-1-7) edge[bend left] node {}(m-1-11)
(m-1-6) edge[bend left] node {}(m-1-10)
(m-1-2) edge[bend left] node {}(m-1-7)      
(m-2-3) edge[bend right] node {}(m-2-4);
\draw(m-3-1.south west) edge (m-3-12.south east);
\end{tikzpicture}
\end{center}
Here the first three rows represent $\mathrm{Tor}^{\Rin}_{*,1}$ (where the first two are the summand (\ref{towersum})) while the forth row represent $\mathrm{Tor}^{\Rin}_{*,2}$. In the general case, the third row will be a summand $\V_{5-4(k+k')}(k')$ and the forth row $\V_{2-4(k+k')}(k')$. In particular, there cannot be non trivial elements in the towers coming from $\mathrm{Tor}^{\Rin}_{*,2}$ for grading reasons, so that the correction terms are determined by $\mathrm{Tor}^{\Rin}_{*,1}$. The result then follows by Poincar\'e duality.
\end{proof}

\begin{proof}[Proof of Proposition \ref{all}]
Because the Poincar\'e homology sphere $P$ is an $L$-space (because of positive scalar curvature, see \cite{KM}), for any homology sphere $Y$ we have
\begin{equation*}
\HSt_{\bullet}(Y\hash P)=\HSt_{\bullet}(Y)\langle-2\rangle
\end{equation*}
because the Eilenberg-Moore spectral sequence collapses at the $E^2$ page. In particular, its effect is to shift all the correction terms by $-1$. Hence the result follows by considering $Y_k\hash Y_{k'}$ and adding suitably many copies of $P$ and $\bar{P}$.
\end{proof}

The same argument in the proof of Lemma \ref{sum} can be readily applied to connected sums of many homology spheres of simple type $M_k$. While the correction terms will only depend on the two maximal indices, one can see that the $\Rin$-module structure can become arbitrarily complicated as the number of summands grows. For example, this shows that  for each $N$ there exist homology spheres $Y$ such that $\HSt_{\bullet}(Y)$ satisfies the following. There exists sequences $\mathbf{x}_i$, $\mathbf{y}_i$ for $1\leq i\leq N$ not in the image of $i_*$ such that the following hold:
\begin{itemize}
\item $V\cdot \mathbf{x}_i=\mathbf{y}_i$;
\item $Q\cdot \mathbf{x}_1\in \mathrm{Im}(i_*)$;
\item $Q\cdot \mathbf{x}_i=\mathbf{y}_{i-1}$ for $i\geq 2$.
\end{itemize}
As every homology sphere is \textit{invertibly} cobordant to a hyperbolic one (\cite{AKMR}), and by funtoriality invertible cobordisms induce injective maps in Floer homology, we can also choose such examples to be hyperbolic.

\vspace{0.5cm}
\bibliographystyle{alpha}
\bibliography{biblio}

\end{document}